\documentclass{amsart}

\usepackage{amsthm}
\usepackage{amsbsy,amsmath,amssymb,amscd,amsfonts,float}

\usepackage{graphicx,xcolor}			% For including pictures
\usepackage[pagebackref=true,colorlinks=true, urlcolor=blue, citecolor=black, linkcolor=black]{hyperref}  
\usepackage[export]{adjustbox}
\usepackage[font={small}]{caption}
\usepackage{subcaption}			% For captioning multi-panel figures

\usepackage[nameinlink,capitalize,noabbrev]{cleveref}
\usepackage{paralist}

\newtheorem*{theorem*}{Theorem}
\newtheorem*{result*}{Result}

\newtheorem{proposition}{Proposition}
\newtheorem{conjecture}{Conjecture}
\newtheorem{corollary}{Corollary}
\newtheorem{lemma}{Lemma}
\theoremstyle{remark}
\newtheorem{remark}{Remark}
\theoremstyle{definition}
\newtheorem{definition}{Definition}

\hypersetup{
    pdftoolbar=true,        % show Acrobat’s toolbar?
    pdfmenubar=true,        % show Acrobat’s menu?
    pdffitwindow=false,     % window fit to page when opened
    pdfstartview={FitH},    % fits the width of 
    colorlinks=true,       % false: boxed links; true: colored links
    linkcolor=olive,          % color of internal links (change box color with linkbordercolor)
    citecolor=blue,        % color of links to bibliography
    filecolor=black,      % color of file links
    urlcolor=red           % color of external links
}

\newcommand{\Q}{\mathcal{Q}}
\newcommand{\A}{\mathcal{A}}
\newcommand{\M}{\mathcal{M}}

\newcommand{\K}{\mathcal{K}}

\newcommand{\T}{\mathcal{T}}
\newcommand{\C}{\mathcal{C}}
\newcommand{\F}{\mathcal{F}}

\renewcommand{\P}{\mathcal{P}}

\renewcommand{\T}{\mathcal{T}}
\newcommand{\R}{\mathcal{R}}

\usepackage{esvect}

\newcommand{\torp}[2]{\texorpdfstring{#1}{#2}}

\usepackage{enumerate}
\title[Exploring the Dynamics of the Circumcenter Map]{Exploring the Dynamics of\\the Circumcenter Map\vspace{-.25cm}}

\author[N. McDonald]{Nicholas McDonald}
\thanks{N. McDonald, ETHZ, Lausanne, Switzerland. \texttt{nicholasmcdonald40@gmail.com}}
\author[R. Garcia]{Ronaldo Garcia}
\thanks{R. Garcia, Inst. Mat. Estat., Univ. Fed. Goiás, Goiânia, Brazil. \texttt{ragarcia@ufg.br}}
\author[D. Reznik]{Dan Reznik}
\thanks{D. Reznik$^*$, Data Science Consulting, Rio de Janeiro, Brazil. \texttt{dreznik@gmail.com}}

\date{}	

\begin{document}

\maketitle

\vspace{-1cm}
\begin{abstract}
Using experimental techniques,
we study properties of the ``circumcenter map'', which, upon $n$ iterations sends an $n$-gon to a scaled and rotated copy of itself. We also explore the topology of area-expanding and area-contracting regions induced by this map.
\end{abstract}

\section{Introduction}
\label{sec:intro}
Using simulation and visualization techniques, we explore interesting plane curves implied by a certain map applied to a generic polygon. For a preview of the graphical results, see \cref{fig:regions3456,fig:n3to11}. Let us first define this map, called here the ``circumcenter map''. Referring to \cref{fig:map}:

\begin{definition}[Circumcenter Map]
Given a point $M$ and a polygon $\P$ with vertices $P_1,\ldots,P_n$, the {\em circumcenter map} $\C_M(\P)$ yields a new polygon $\P'$ whose vertices are the circumcenters of $MP_1P_2$, $MP_2P_3$, \ldots, $MP_{n}P_1$, respectively.
\label{def:circ}
\end{definition}

\begin{figure}[H]
    \centering
    \includegraphics[trim=0 0 0 120,clip,width=.8\textwidth,frame]{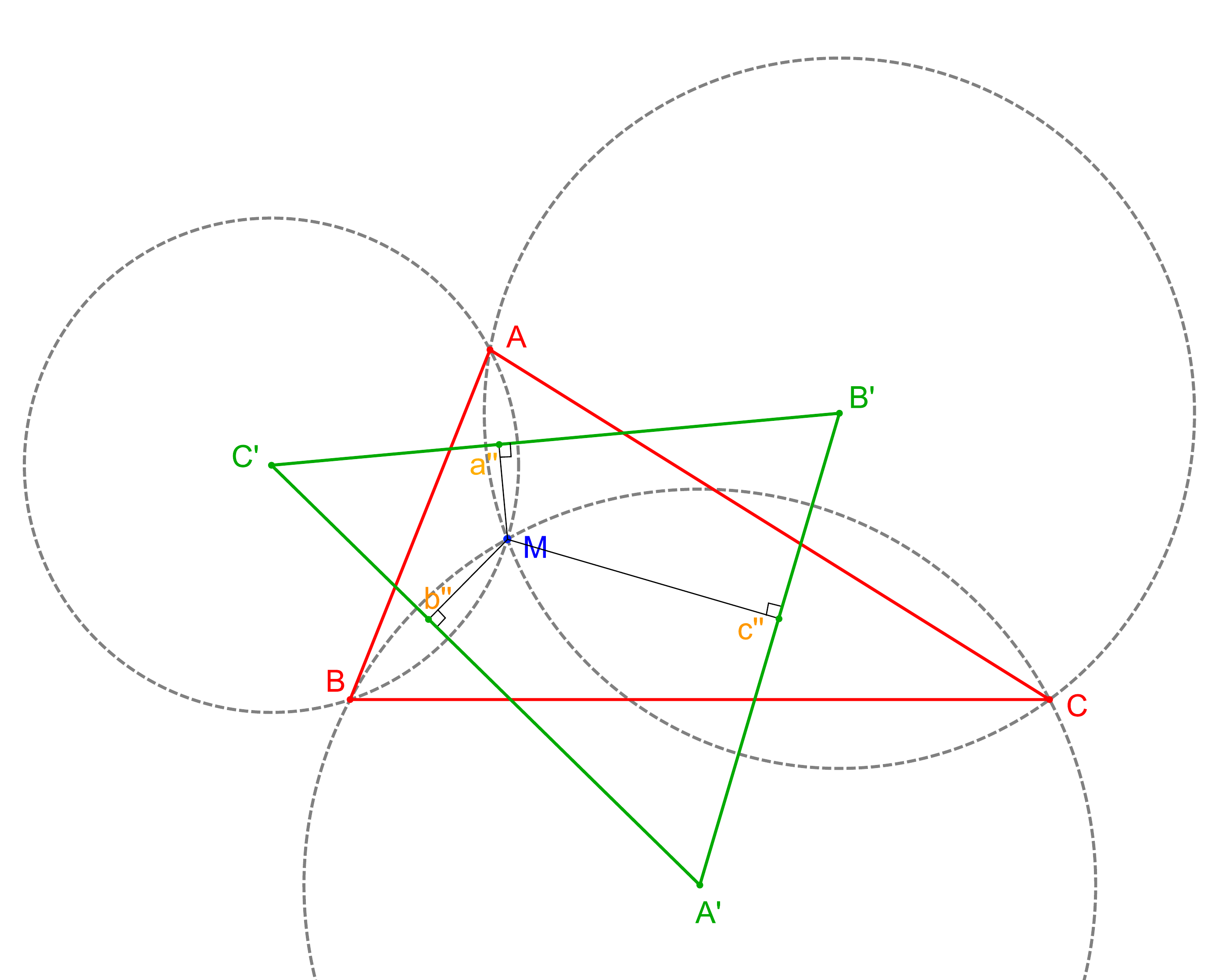}
    \caption{Given a point $M$, the circumcenter map sends a triangle $T=ABC$ to $T'=A'B'C'$ with vertices at the circumcenters of $MBC$, $MCA$, $MAB$, respectively. Also shown are the vertices $a'',b'', c''$ of the $M$-pedal of $T'$ which is a half-sized version of $T$.}
    \label{fig:map}
\end{figure}

We review classical results that underpin the phenomena studied herein, namely:
\begin{enumerate}[i)]
\item $\C_M(\P)$ yield the (half-sized) $M$-antipedal polygon, see \cref{fig:5-gon};
\item as a consequence of a result proved in \cite{stewart1940}, $n$ consecutive applications of this map, i.e., $\C_M^n(\P)$ yields a new polygon which is the image of $\P$ under a rigid rotation about $M$ by $\alpha$, and a homothety (uniform scaling) with ratio $s$, see \cref{fig:n4n5-reg};
\item subsequent $n$ applications of the map result in a transformation with identical parameters $\alpha$ and $s$. 
\end{enumerate}

\begin{figure}
    \centering
    \includegraphics[trim=25 50 0 0,clip,width=.8\textwidth,frame]{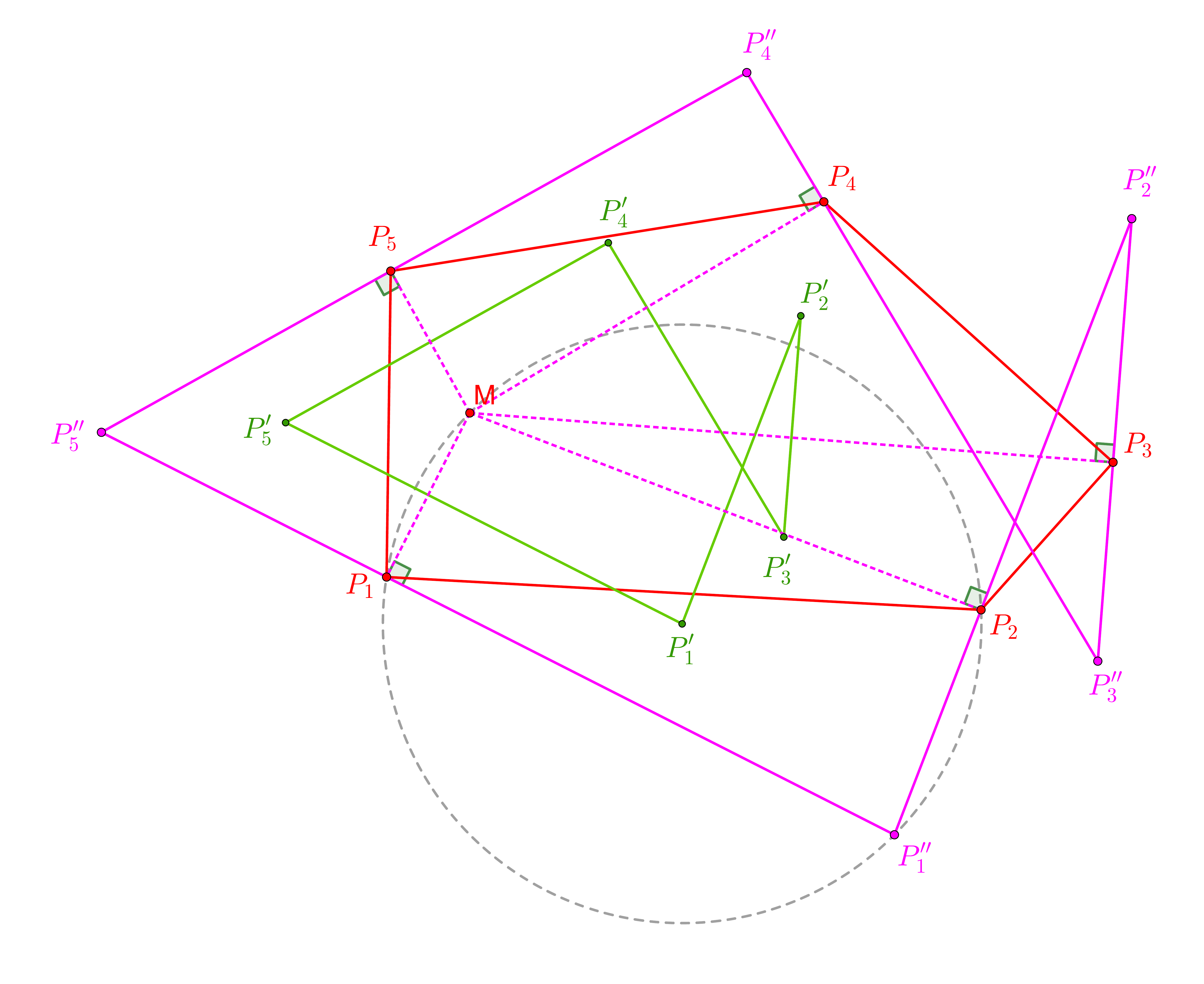}
    \caption{If $\P=\{P_i\}$ is a 5-gon (red), $\C_M(\P)$ is another 5-gon (green) which is a half-sized version of the $M$-antipedal of $\P$ (magenta). Also shown is the construction for $P_1'$ which is the center of circle $M P_i P_{i+1}$ (dashed gray).}
    \label{fig:5-gon}
\end{figure}

Given a starting polygon $\P$, we will study the locus of $M$ such that after $n$ applications of the map the resulting polygon has the same area ($s=1$) and/or angle ($\alpha=0$) as the original one. Experiments reveal an interesting structure of such loci, and beautiful symmetries when $\P$ is regular, see \cref{fig:n4n5-shrink}.

\begin{figure}
    \centering
    \includegraphics[width=\textwidth]{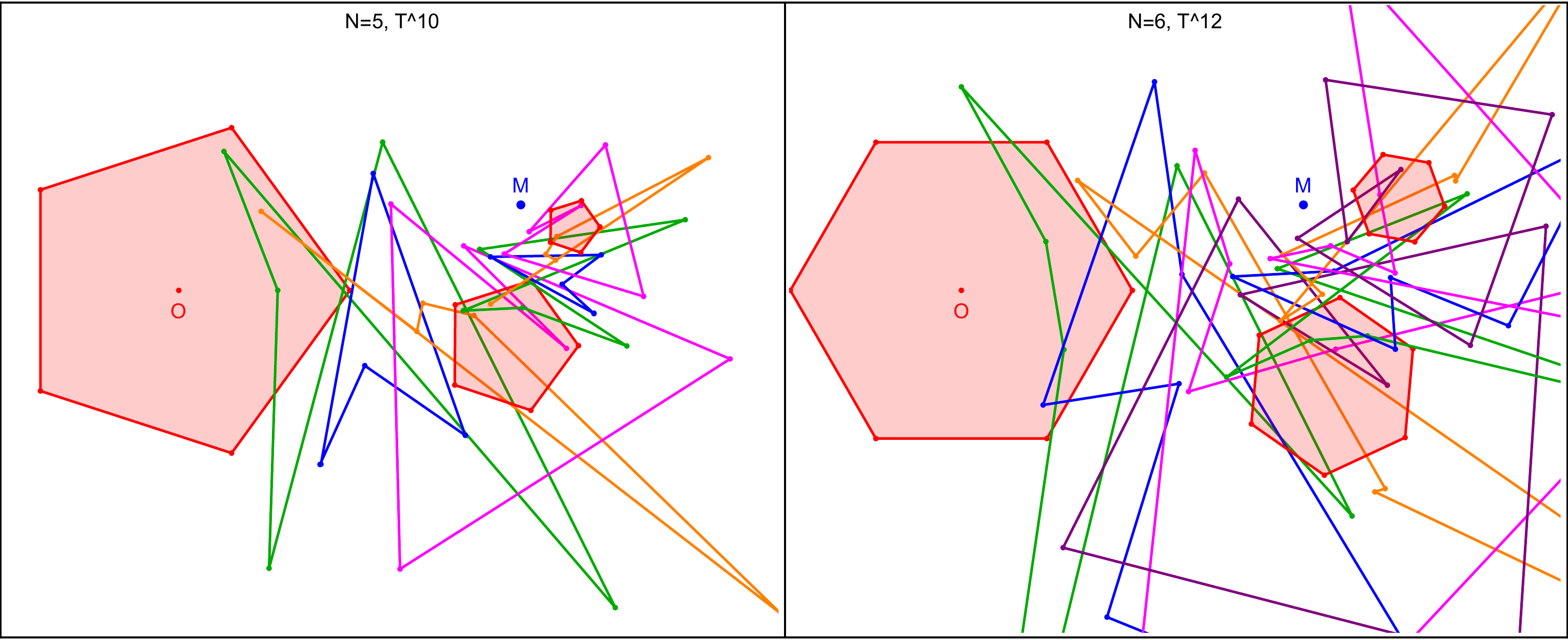}
    \caption{\textbf{Left:} 10 applications of the circumcenter map on a regular pentagon (left), showing 5-periodicity.  \textbf{Right:} 12 applications starting from a regular hexagon. Note intermediate shapes are arbirtrary, see \cref{fig:aliens456}.}
    \label{fig:n4n5-reg}
\end{figure}

\begin{figure}
    \centering
    \includegraphics[width=\textwidth]{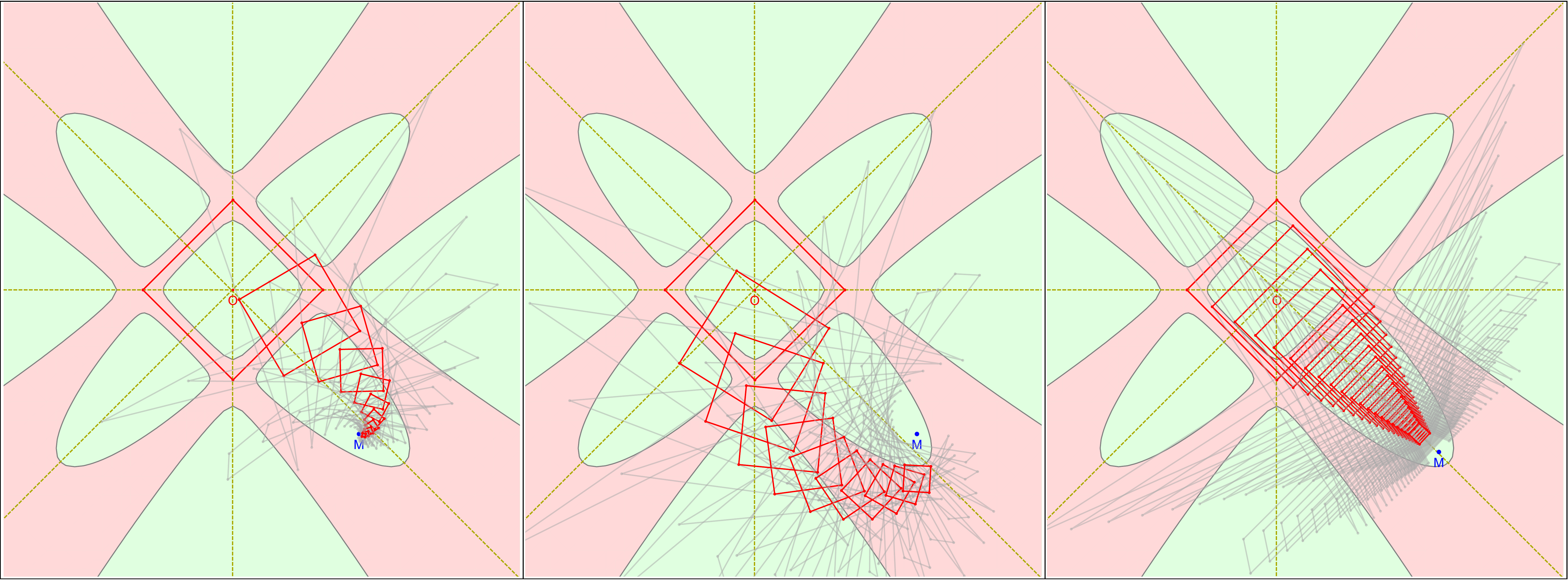}
    \caption{Iterations of the circumcenter map, starting from a square ($n=4$, red), with $M$ at an area-contracting region (green), but in locations which induce negative (left), positive (middle), or zero rotation (right) on the sequence. Intermediate iterates are shown gray, and can be of diverse shapes. Every four iterations produces a new polygon (red) which is a rotated, scaled version with respect to the original. The locus of $M$ such that $s=1$ (neutral scaling) is the boundary between the red and green regions; its shape depends on the original polygon.}
    \label{fig:n4n5-shrink}
\end{figure}

\subsection*{Main results}

The claims below were initially evidenced by simulation. Most are proved with a computer algebra system (CAS) and their proofs will be omitted.

\begin{enumerate}[1)]
\item We explicitly derive the locus of $M$ such that $s=1$, when $\P$ is an equilateral triangle (resp. a square): it is a 6-degree (resp. 8-degree) polynomial on the coordinates of $M$.
\item If $\P$ is a regular $n$-gon, the locus of $M$ such that $\alpha=0$ is the union of $n$ lines through the centroid of $\P$, rotated about each other by $\pi/n$.
\item In all cases there is a discrete set of locations such that both $s=1$ and $\alpha=0$. We study the map with $\P$ an equilateral triangle, showing that if $M$ is on any one of these locations the map is 3-periodic. If $M$ is at the centroid of the equilateral, the map is 2-periodic.
\item Based on compactified visualizations of the $s=1$ boundary, we conjecture that for all $n$: (i) there is only one connected region such that $s>1$, and (ii) if $n$ is odd (resp. even), the number of connected regions such that $s<1$ is given by $1+n(n+1)/2$ (resp. $1+n^2/2$), i.e., in both cases it is of $O(n^2)$.
\item We informally investigate topological changes in the $s=1$ locus with respect to affine stretching of the initial polygon.
\end{enumerate}

\subsection*{Background}

%The {\em pentagram map} sends a polygon to another one whose vertices lie along intersections of certain diagonals, with two applications resulting in a similar figure  \cite{ovsienko2010-pentagram}. A map based on reflections of polygon vertices about incident sides is studied in \cite{maxim2021-centroaffine}.
%The {\em circumcenter-of-mass} of a polygon is the area-weighted average of circumcenters of triangles in a triangulation whose location is independent of the triangulation itself, see \cite{tabachnikov2015-com}.

In \cite{stupel2018-subcevians}, properties of the ``central'' sub-triangle defined by a four-fold subdivision of a reference triangle (using cevians) are studied. A map based on the 2nd isodynamic point of a polygon's subtriangles is described in \cite{reznik2021-wolfram-iso}.

\subsection*{Article Organization}

In \cref{app:stewart} we review classical results that show that $n$ applications of the circumcenter map are a similarity transform with parameters $s$ and $\alpha$ which only depend on $M$. In \cref{sec:n3} we describe properties of the map for the initial polygon a triangle or a square. In \cref{sec:all-n} we extend the analysis to regular polygons of any number of sides. Conclusions and suggestions for further exploration appear in \cref{sec:conclusion}. To facilitate reproducibility, explicit expressions for the circumcenter map appear in \cref{app:map-tri}.

\section{Review: Stewart's Result}
\label{app:stewart}
Referring to \cref{fig:pedals}, recall that (i) the $M$-pedal polygon of a polygon $\P$ has vertices at the intersections of perpendiculars dropped from $M$ onto the sides of $\P$; (ii) the $M$-antipedal polygon of $\P$ is such that $\P$ is its $M$-pedal. Finally, (iii) the $M$-reflection polygon of $\P$ has vertices at the reflections of $M$ about the sidelines of $\P$. Clearly:

\begin{remark}
the $M$-reflection polygon of $\P$ is the twice-sized $M$-pedal polygon of $\P$, with $M$ as the homothety center.
\label{rem:twice}
\end{remark}

\begin{figure}
    \centering
    \includegraphics[width=.8\textwidth]{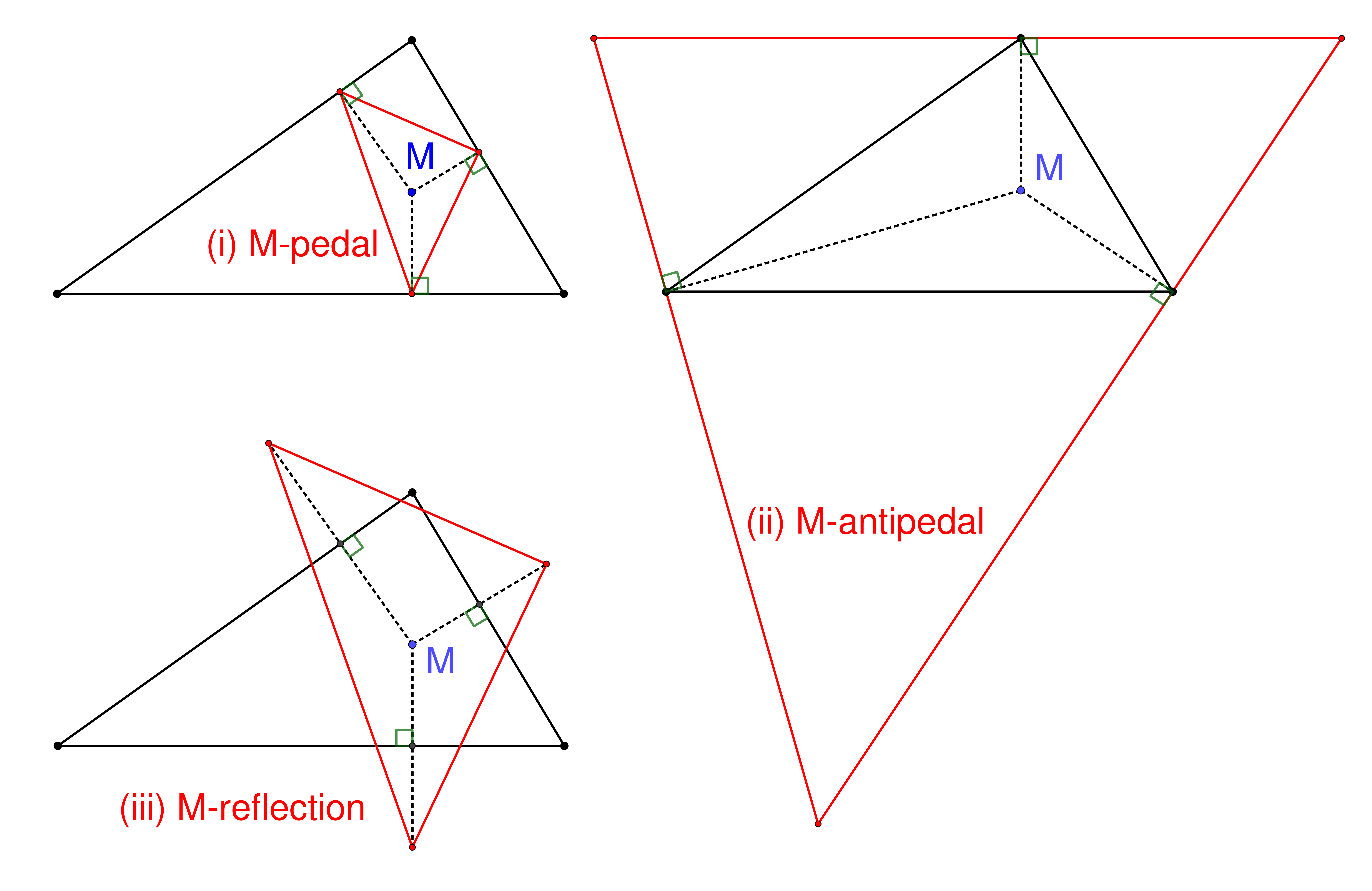}
    \caption{Constructions of pedal, antipedal, and reflection polygons, for the case of a triangle.}
    \label{fig:pedals}
\end{figure}

Let $\P'=\C_M(\P)$ be the polygon obtained under the circumcenter map of \cref{def:circ}. Let $P_i',i=1\ldots n$ denote its vertices.

\begin{lemma}
The polygon $\C_M^{-1}(\P)$ is the $M$-reflection polygon $\P$.
\label{lem:inverse}
\end{lemma}

\begin{proof}
Referring to \cref{fig:map},  $P_i'$ (resp. $P'_{i+1}$) is the center of a circle $\K_i$ passing through $M,P_i, P_{i+1}$ (resp. $\K_{i+1}$ passing through $M,P_{i+1},P_{i+2}$). 
We have that $\K_i\cap \K_{i+1}=\{M,P_i\}$.

The inverse circumcenter map is 
$\C_M^{-1}(\P')= 
 \P=\{P_i\}$. For each such pair of consecutive circles, the intersections $\{M,P_i\}$ are symmetric about $P'_i P'_{i+1}$. See \cref{fig:map} for the case $n=3$.
\end{proof}

\noindent Using \cref{rem:twice,lem:inverse}, the following property is illustrated in \cref{fig:5-gon}:
%, The inverse of the circumcenter map sends $\P'$ to the twice-sized $M$-pedal of $\P'$, and therefore, 

\begin{corollary}
$\C_M(\P)$ is homothetic to the $M$-antipedal of $\P$, with ratio $1/2$ and homothety center $M$.
\label{cor:antipedal}
\end{corollary}
\noindent Recall a well-known result by Johnson \cite[Theorem 2c]{johnson1918-similar}: 

\begin{result*}[Johnson, 1918]
If two polygons $\F$ and $\F'$ with no parallel sides are similar, there exists a point $M$, called the {\em self-homologous point}, such that $\F'$ is a rigid rotation of $\F$ about $M$ followed by uniform scaling about the same point.
\end{result*}

\noindent Using a definition in \cite[Construction 1]{stewart1940}:

\begin{definition}[Miquel Map]
Given a point $M$ and an angle $\theta$, let $\M$ denote a map that sends a polygon $\P=\{P_i\}$, to a new polygon $\P'$ (known as the Miquel polygon) with each vertex $P'_i$ is on the line $P_i P_{i+1}$, and such that $\measuredangle P_i P_i' M = \theta$, for $i=1,\ldots,n$. 
%$\P'(M,\theta)$ denote the Miquel polygon of $\P$. It has vertices $P_i'$ which lie on $P_i P_{i+1}$ and such that $\measuredangle P_i P_i' M = \theta$.
\label{def:miquel}
\end{definition}

\noindent Let $\M^k$ denote $k$ successive applications of the Miquel map. The following key result was proved in \cite[Theorem 2]{stewart1940}:

\begin{theorem*}[Stewart, 1940]
Let $\P$ be a polygon with $n$ sides. $\M^n(\P)$ is similar to $\P$ with $M$ as the self-homologous point. If follows that if $\M^j(\P)$ is similar to $\M^i(\P)$, then   $i \equiv j~(mod~n)$. 
\end{theorem*}

\noindent Let $\M_\perp$ denote the Miquel map when $\theta=\pi/2$. \cref{def:miquel} implies that $\M_\perp(\P)$ is the pedal polygon of $\P$ with respect to $M$. Referring to \cref{fig:iterations}:

\begin{corollary}
$\M_\perp^n(\P)$ is self-homologous to $\P$, i.e., it is an image under a rigid rotation and a  uniform scaling of $\P$ about $M$. 
\end{corollary}

\noindent The inverse map $\M_\perp^{-1}(\P)$ yields the $M$-antipedal polygon of $\P$. Referring to \cref{cor:antipedal}:

\begin{remark}
$\C_M(\P)$ is a half-sized homothety of $\M_\perp ^{-1}(\P)$ with respect to $M$.
\end{remark}

Since the pedal transformation has an inverse,
$\M_\perp^{-n}(\P)$ is similar to $\P$. Referring to \cref{fig:iterations}, and noting that the scaling in \cref{cor:antipedal}  does not affect similarity:

\begin{corollary}
$\C_M^n(\P)$ is similar to $\P$.
\label{cor:similar}
\end{corollary}

\begin{figure}
    \centering
    \includegraphics[width=\textwidth]{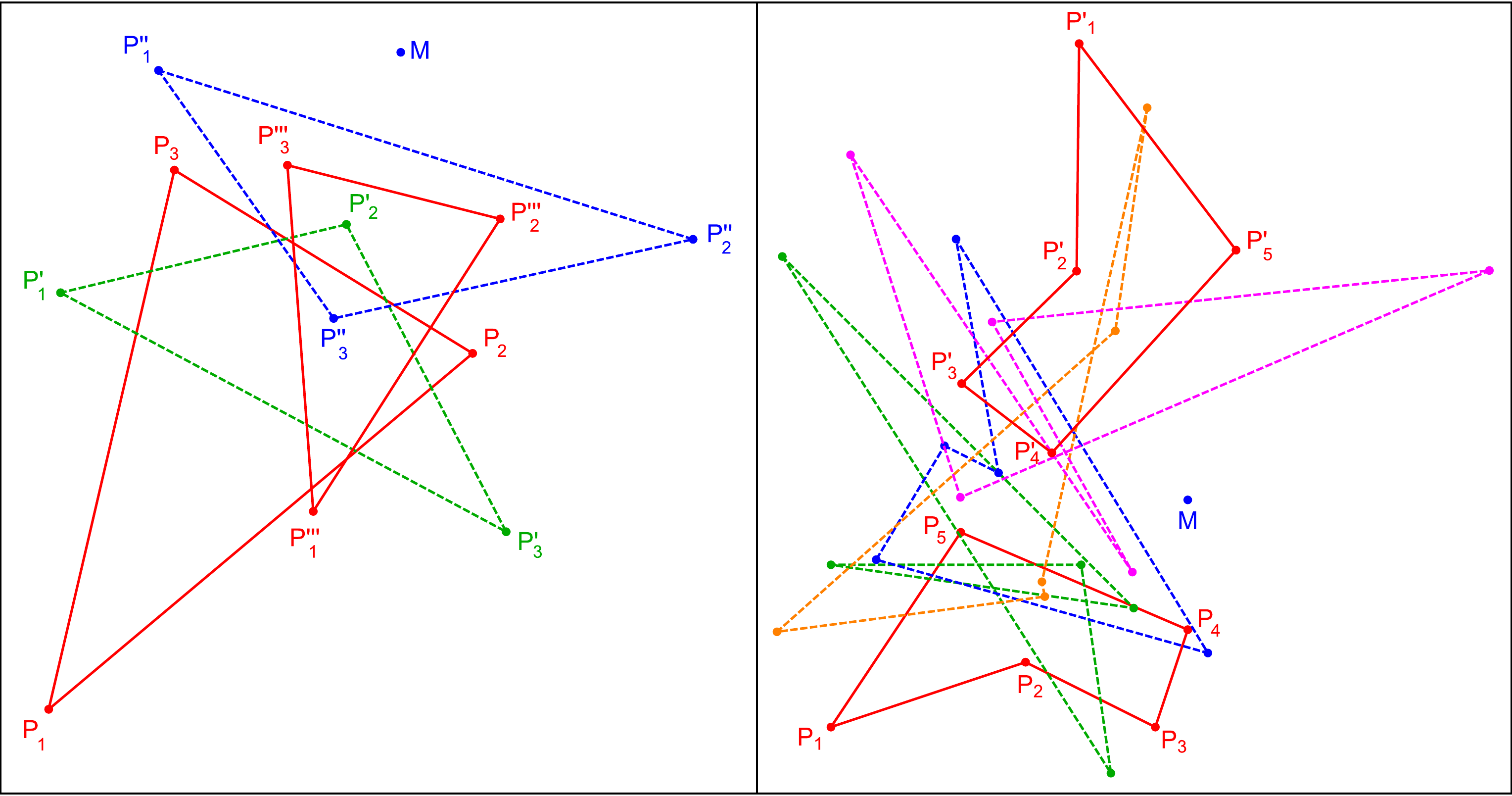}
    \caption{\textbf{Left:} Applying the circumcenter map with respect to $M$ to a starting triangle $P_1 P_2 P_3$ (solid red) yields the (dashed green) triangle $P_1' P_2' P_3'$. In turn, applying the map on the latter yields the (dashed blue) triangle $P_1'' P_2'' P_3''$. Finally, a third application of the map yields $P_i'''$ (solid red), similar to the original. \textbf{Right:} Likewise, starting from a pentagon $P_i$, i=$1,\ldots,5$, $n=5$ applications of the map yields $P_i'$ (dashed red), self-homologous to the original. Intermediate generations are colored dashed green, blue, orange, and magenta.}
    \label{fig:iterations}
\end{figure}

Let $\P'=\C_M^{n}(\P)$, and $\P''=\C_M^{2n}(\P')$. Express these as $ \P'=\T( \P)$, and $ \P''=\T'(\P')$, where $\T,\T'$ are similarity transforms (that is, composition of a rotation and   scaling about $M$).
\begin{proposition}
$\T=\T'$.
\end{proposition}
%, i.e., every $n$ iterations of the map produces a similar polygon and their centroids will lie on logarithmic spiral.
\begin{proof}
Since $\P'= \T(\P)$, then $\C_M^n(\P') = \C_M^n(\T (\P))$. By definition (see \cref{sec:intro}), the circumcenter map is based on the circumcenters of subtriangles of a given triangle. Since circumcenters are {\em triangle centers} (see \cite{kimberling1993_rocky}), they are equivariant over similarity transforms, therefore, $\C_M^n(\T( \P))=\T(\C_M^n(P))=\T^2(\P)$, i.e., $\C_M^n(\P')=\T(\P')$. 
\end{proof}

In \cref{fig:n4n5-reg,fig:aliens456}, we illustrate a statement by Stewart in \cite{stewart1940}, namely, that  intermediate applications of the map produce  polygons ``as diverse in shape as is imaginable''.

\begin{figure}
    \centering
    \includegraphics[width=.8\textwidth]{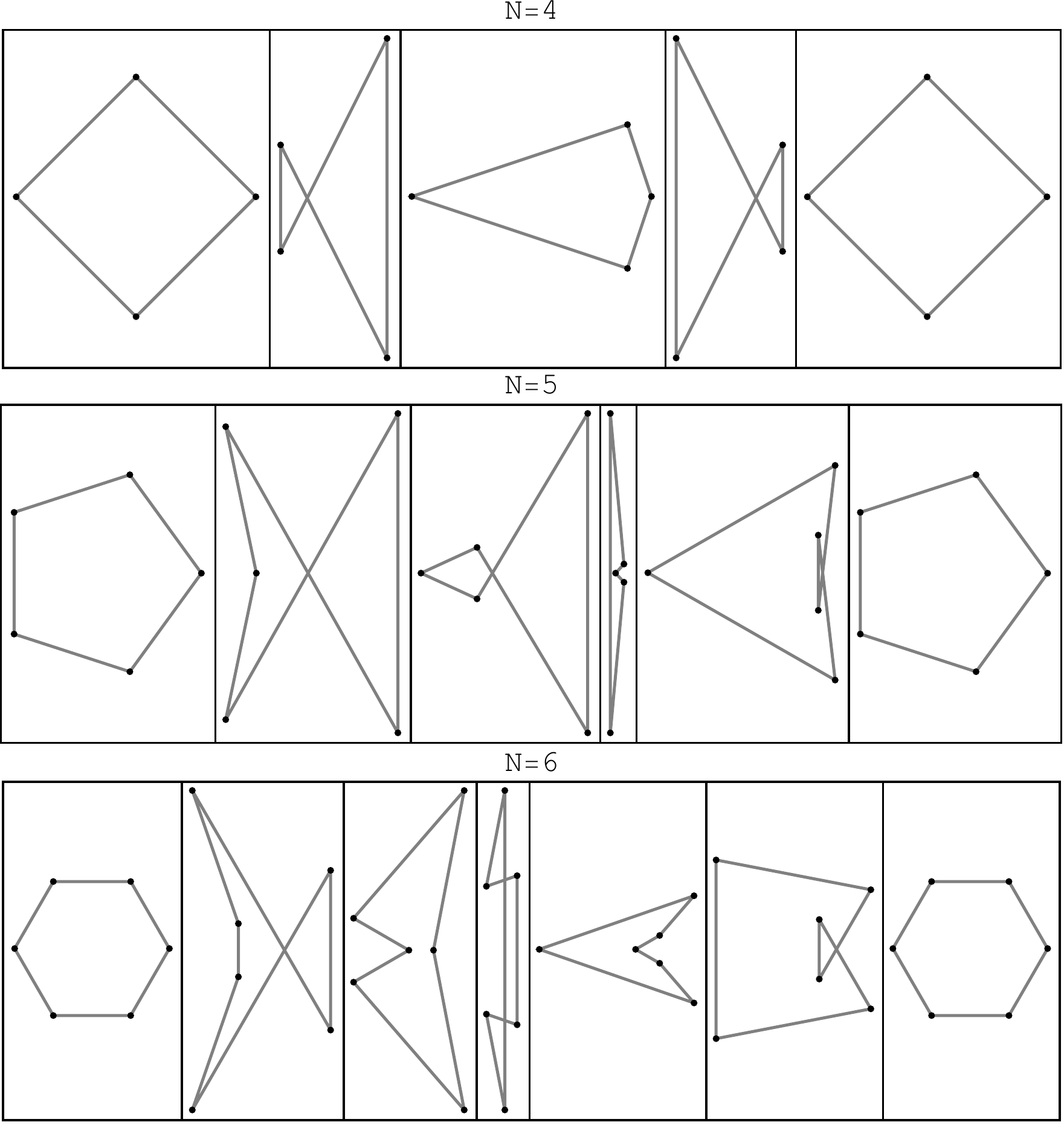}
    \caption{Intermediate shapes (with scale and rotation removed) produced by sequential applications of the circumcenter map. Top, middle, bottom strips show sequence starting from square, regular pentagon, and regular hexagon, respectively. In each case, the rightmost vertex of the starting polygon is at $(1,0)$, and $M$ (not shown) is at $(2,0)$.}
    \label{fig:aliens456}
\end{figure}

\section{The \torp{$n=3$}{n=3} and \torp{$n=4$}{n=4} cases}
\label{sec:n3}
In this section we study the locus of $M$ such that $\C_M^n(\P)$ is  area-preserving ($s=1$) and/or rotation-neutral ($\alpha=0$), for the cases where $\P$ is an equilateral or a square.

Let $\P=ABC$ be a triangle, and 
%$\P'$ result from 3 consecutive applications of the circumcenter map, i.e.,
$\P'=\C_M^3(\P)=A'B'C'$. Let $\A(\Q)$ denote the area of a polygon $\Q$. Via CAS, we obtain the following proposition.
\begin{proposition}
The ratio of sides $A'B'$, $B'C'$ and $A'C'$ of $\P' $ and $\P$ is given by:
\begin{align*}
 \frac{|A'B'|}{|AB|}=\frac{|A'C'|}{|AC|}=\frac{|B'C'|}{|BC|} =\frac{{l_a l_b l_c m_a m_b m_c}}{8~\A(ABM)\A(BCM)\A(ACM)}
\end{align*}
where $l_a=|BC|$, $l_b=|AC|$, $l_c=|AB|$, $m_a=|AM|$, $m_b=|BM|$, $m_c=|CM|$.
\end{proposition}

%\begin{proposition}
%The ratio of sides of $\P'=A'B'C'$ and $\P$ is given by:
%\begin{align*}
 %\frac{|A'-B'|}{|A-B|}=\frac{|A'-C'|}{|A%-C|}=\frac{|B'-C'|}{|B-C|} =\frac{{l_a %l_b l_c m_a m_b %m_c}}{8~\A(ABM)\A(BCM)\A(ACM)}
%\end{align*}
%where $l_a=|B-C|$, $l_b=|C-A|$, %$l_c=|A-B|$, $m_a=|A-M|$, $m_b=|B-M|$, %$m_c=|C-M|$.
%\end{proposition}
%comentario{calculo conferido e é 8 mesmo. Na razao de area é 64}
\begin{proposition}
Let $\alpha$ denote the angle of rotation of $\P'$ with respect to $\P$. Then:
\[ \cos\alpha=\frac{ m_c^2(m_a^2   + m_b^2  )\A(ABM) +m_b^2 (m_a^2   +   m_c^2)\A(ACM)  +m_a^2 (  m_b^2 +   m_c^2)\A(BCM)}{l_al_bl_c m_am_bm_c}\]
 
%\begin{align*}
%na&=(  \frac{1}{2} m_a m_b +  \frac{1}{2} m_a m_c - m_b m_c- \frac{1}{2} l_a m_a ) abm +\frac{1}{2} m_a (l_b + l_c - 2 m_a + m_b + m_c) bcm \\
%&+ (  \frac{1}{2} m_a m_b +  \frac{1}{2} m_a m_c - m_b m_c- \frac{1}{2} l_a m_a ) acm\\
%da&=\sqrt{l_al_bl_c m_am_bm_c}
%\end{align*}
\end{proposition}

\subsection*{The case of an equilateral triangle} Let $\R=ABC$ be the equilateral triangle with vertices $A=(1,0)$,  $B=(-1,\sqrt{3})/2$, $C=(-1,-\sqrt{3})/2$. Let $\R'=\C_M^3(\R)$ with $M=(x_m,y_m)$. Let $s=\A(\R')/\A(\R)$. Via CAS, we obtain the following proposition.  

%By setting $\A(T^3(T_0))=\A(T_0)$

\begin{proposition}
If $\P$ is an equilateral, $s=1$ if $M=(x_m,y_m)$ satisfies:
\begin{align*}
  & 3\,x_m^{6}-
y_m^{6}-12\,x_m^{5}+9\,y_m^{4}+
\left( -27\,y_m^{2}+9
 \right) x_m^{4}+ \left( 24\,y_m^{2}+6 \right) x_m^{3}+\nonumber\\& \left( 33\,y_m^{4}+18\,y_m^{2}-6 \right) x_m^{2}+
 \left( 36\,y_m^{4}-18\,y_m^{2} \right) x_m-6\,y_m^{2}=0
\end{align*}
\label{prop:equi6}

\end{proposition}
\cref{fig:equi-its,fig:n3-neutral}, illustrate sequential applications of the circumcenter map starting from an equilateral, for the cases of $M$ in an area-contracting, area-expanding, and the boundary in between them defined in \cref{prop:equi6}.

\begin{proposition}
If $\P$ is an equilateral, $\alpha=0$ when $M=(x_m,y_m)$ satisfies:
\[y_m(y_m^2-3x_m^2)=0\]
\label{prop:alpha0-3}
\end{proposition}

\begin{figure}
    \centering
    \includegraphics[width=\textwidth]{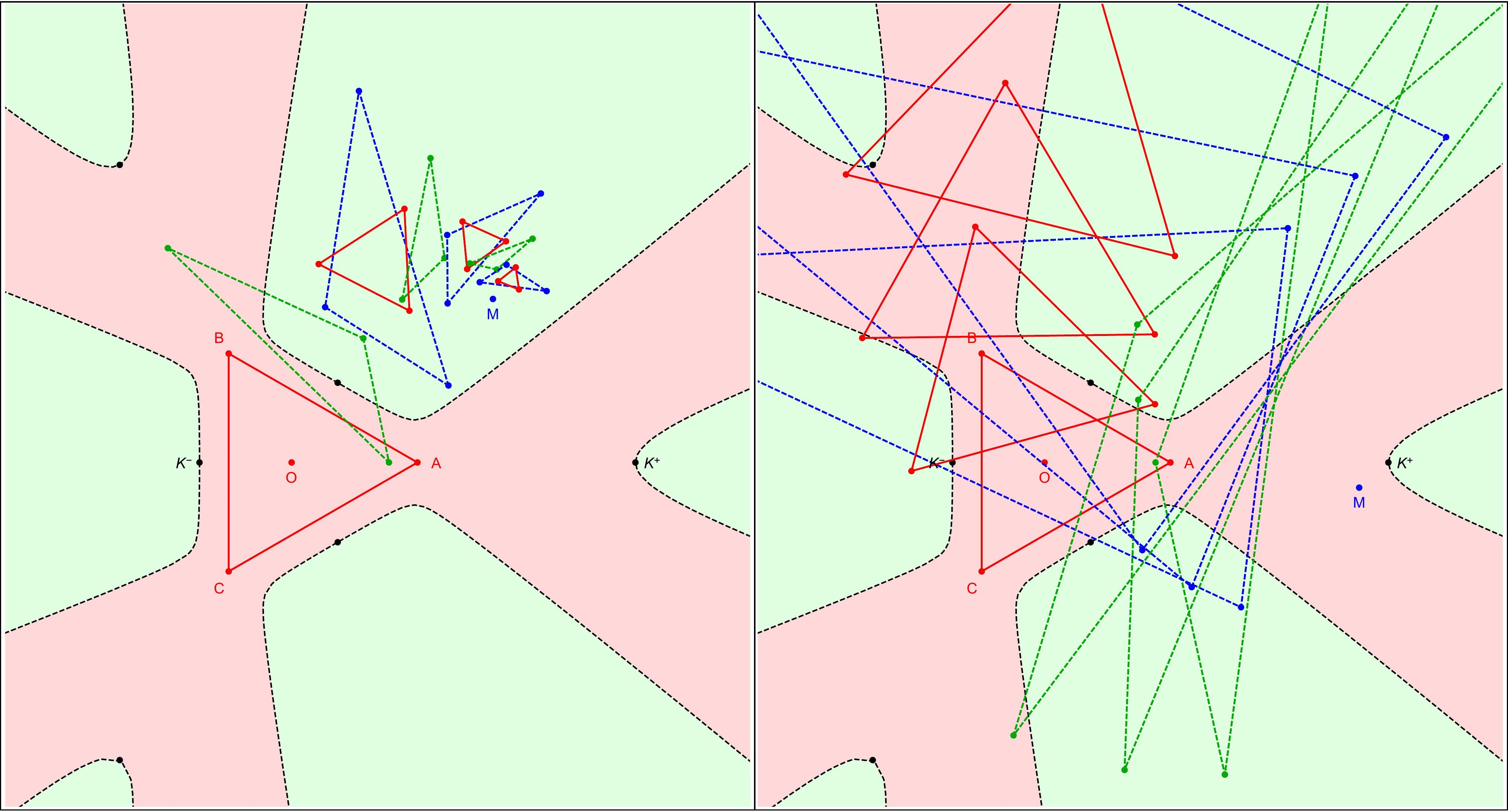}
    \caption{\textbf{Left:} Nine iterations of the circumcenter map 
    starting from an equilateral triangle (solid red) centered at $O$. $M$ is placed in an area-shrinkage region (green). A new, smaller equilateral (red) is produced every 3 applications of the map. \textbf{Right:} with $M$ on the area-expansion region (green), the area expands upon every 3 applications of the map.}
    \label{fig:equi-its}
\end{figure}

\begin{figure}
    \centering
    \includegraphics[trim=0 0 583 0,clip,width=.6\textwidth]{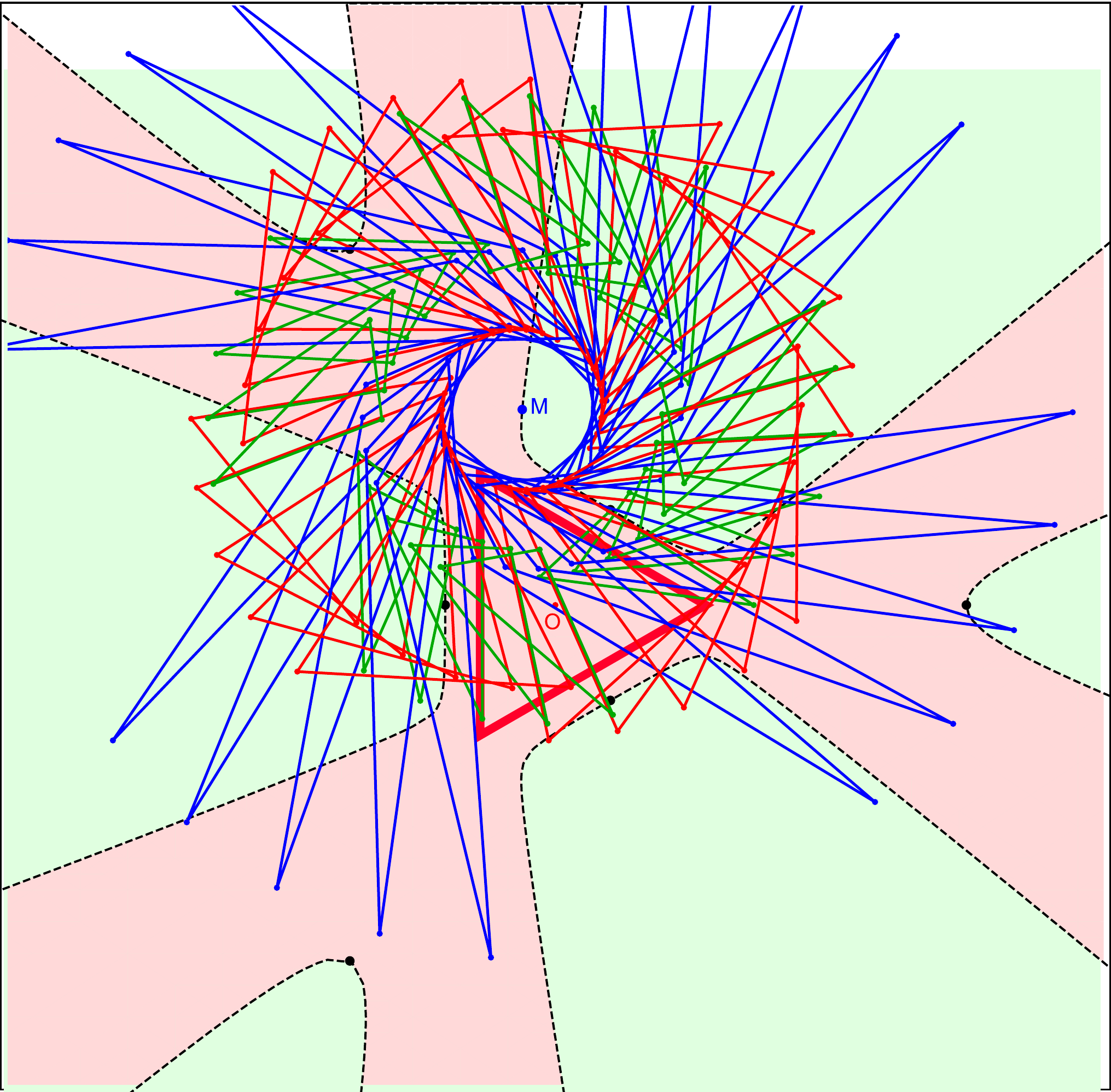}
    \caption{Starting from an equilateral (thick red) with centroid $O$, an $M$ is selected on the boundary between area-expansion (light red) and area-contraction (light green) regions. Sequential applications of the circumcenter map are shown in green, blue, and back to red. Since $M$ is on the boundary, every three applications of the map are area-preserving and result in a constant net rotation about $M$.}
    \label{fig:n3-neutral}
\end{figure}

As shown in \cref{fig:n3-isosceles}, as the starting triangle is affinely distorted, the number of connected components in the $s=1$ locus changes, revealing a non-trivial relationship (not studied here).

\begin{figure}
    \centering
    \includegraphics[width=\textwidth]{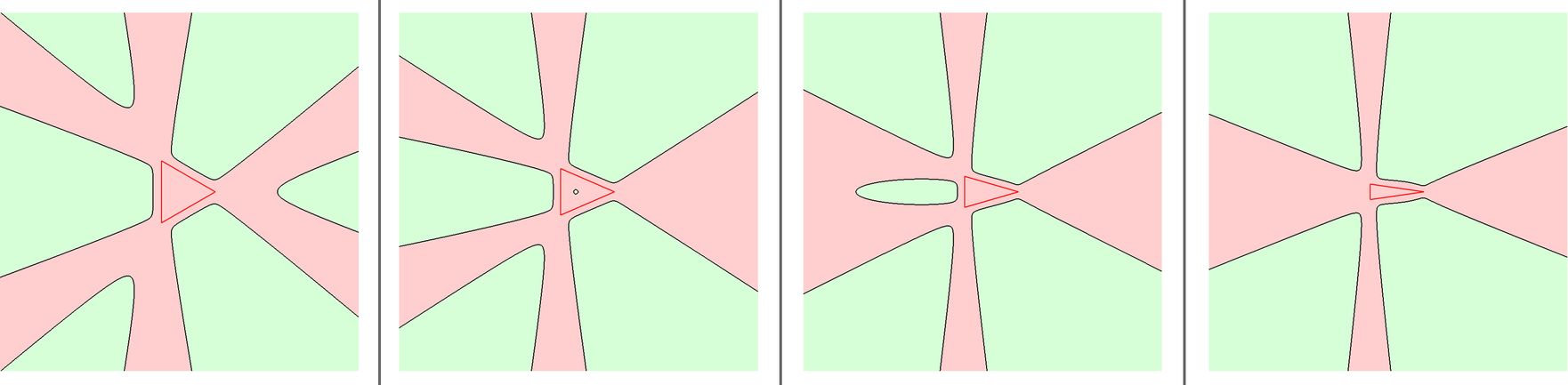}
    \caption{As the starting equilateral (red, left) is stretched horizontally, the $s=1$ boundary (under 3 applications of the circumcenter map) changes topology. See \href{https://bit.ly/3Pl2tmz}{\texttt{bit.ly/3Pl2tmz}} for interactive examples.}
    \label{fig:n3-isosceles}
\end{figure}

Referring to \cref{fig:canonical}:

%\begin{proposition}

\begin{corollary}
There are six points on \cref{prop:equi6} such that $\alpha=0$. These are: $K_1=(1+\sqrt{3},0)$, $K_2=(1-\sqrt{3},0)$, and their rotations by $\pm2\pi/3$.\end{corollary}
%\end{proposition}

Let $\R_i=A_i B_i C_i$ denote the image of $\R$ under $i$ iterations of the circumcenter map. As shown above $\R_3$ is similar to $\R$. Referring to \cref{fig:canonical}(left):

\begin{proposition}
If $M=K_1$, then the vertices of $\R_1$ and $\R_2$ have coordinate 
%$\R_3=\R$ and:
\begin{align*}
A_1=&(1,0),\;\;\; \hskip 1.55cm B_1=\left(1+ \frac{ \sqrt{3}}{2},\frac{3}{2}+\sqrt{3}\right),\;\;\;C_1=\left(1+ \frac{ \sqrt{3}}{2},-\frac{3}{2}-\sqrt{3}\right)\\
A_2=&\left(-2-\sqrt{3}, 0\right),\;\;\;\; B_2=\left(1+\frac{\sqrt{3}}{2}, \frac{3}{2}\right),\hskip 1.3cm C_2= \left(1+\frac{\sqrt{3}}{2}, -\frac{3}{2}\right).
\end{align*}
$\R_1$ (resp. $\R_2$) has internal angles $30^\circ, 75^\circ, 75^\circ$ (resp. $150^\circ, 15^\circ, 15^\circ)$.
\end{proposition}

Referring to \cref{fig:canonical}(right):

\begin{proposition}
If $M=K_2$, then the vertices of $\R_1$ and $\R_2$ have coordinate 
%$\R_3=\R$ and:
\begin{align*}%\begin{matrix}
A_1=&\left(1,0\right),\;\; \hskip 1cm B_1=\left(1- \frac{ \sqrt{3}}{2},-\frac{3}{2}+\sqrt{3}\right),\;\;\;C_1=\left(1- \frac{ \sqrt{3}}{2},\frac{3}{2}-\sqrt{3}\right)\\
A_2=&\left( \sqrt{3}-2, 0\right),\;\;\;B_2=\left(1-\frac{\sqrt{3}}{2}, \frac{3}{2}\right),\;\;\hskip 1.3cm C_2= \left(1-\frac{\sqrt{3}}{2}, -\frac{3}{2}\right).%\end{matrix}
\end{align*}

$\R_1$ (resp. $\R_2$) has internal angles $30^\circ, 75^\circ, 75^\circ$ (resp. $150^\circ, 15^\circ, 15^\circ $).
\end{proposition}

\begin{figure}
    \centering
    \includegraphics[width=\textwidth]{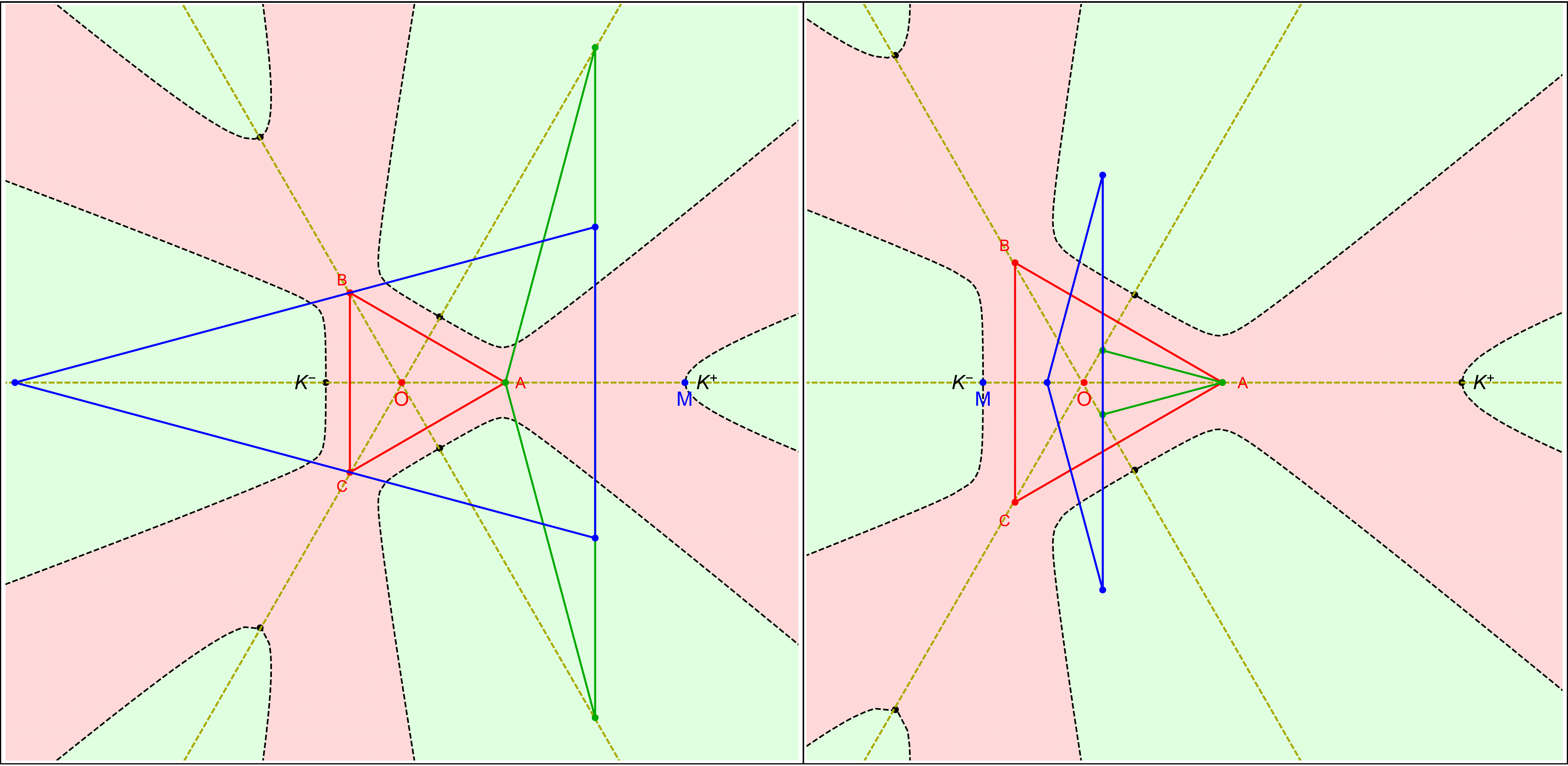}
    \caption{\textbf{Left:} starting with an equilateral $ABC$, when $M$ is at $K^+=(1+\sqrt{3},0)$, repeated applications of the circumcenter map will cycle indefinitely through the 3 canonical triangles shown (red, green, blue). \textbf{Right:} With $M=K^-=(1-\sqrt{3},0)$, the sequence is also 3-periodic, and the canonical triangles obtained are as shown.}
    \label{fig:canonical}
\end{figure}

Referring to \cref{fig:n3-david}, it can be shown that:

\begin{remark}
When $M$ is the centroid of $\R$, repeated applications of the circumcenter are 2-periodic, where the first triangle is $\R$ and the second one is a reflection of $\R$ about said centroid.
\end{remark}

\begin{figure}
    \centering
    \includegraphics[width=\textwidth]{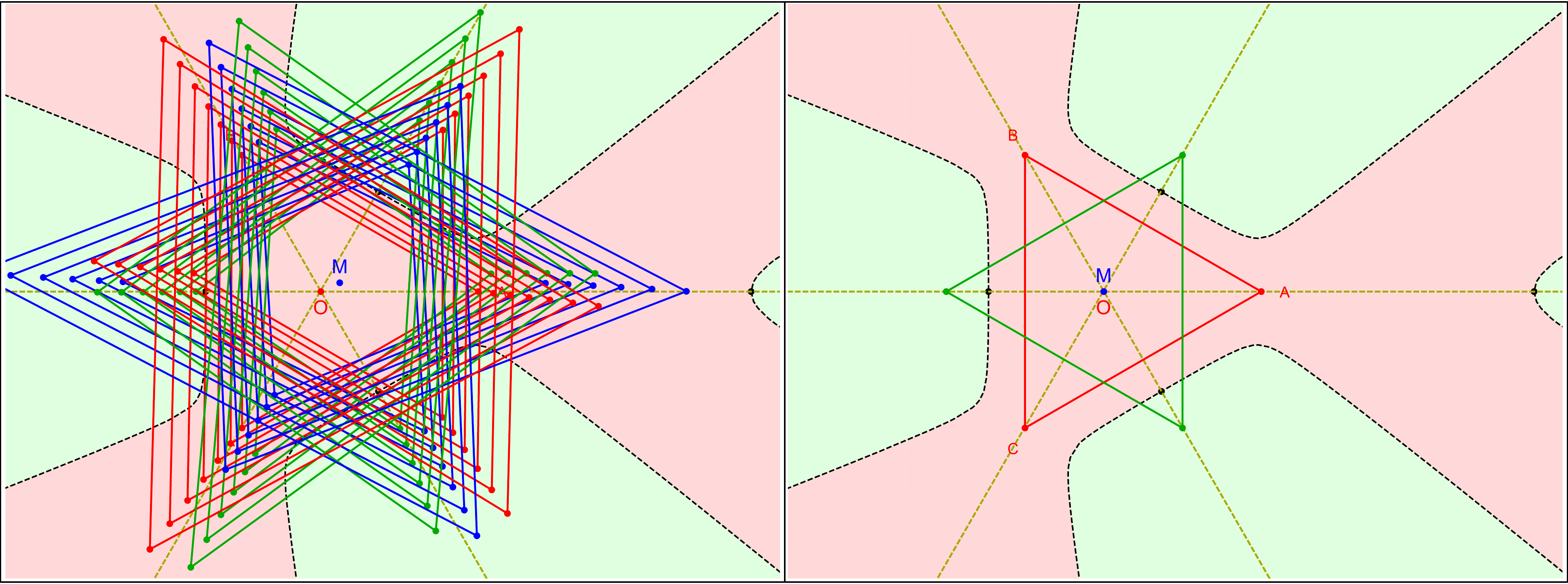}
    \caption{\textbf{Left:} 36 iterations of the circumcenter map with $M$ close to the centroid $O$ of the starting equilateral, resulting in a sequence which is slightly expanding. \textbf{Right:} If $M=O$, the sequence becomes 2-periodic, containing only the original $ABC$ (blue) and its reflection about $O$ (green).}
    \label{fig:n3-david}
\end{figure}

\subsection*{The case of a square}

Let $\Q=ABCD$ be a square with with vertices $A=(1,0)$,  $B=(0,1)$, $C=(-1,0)$, $D=(0,-1)$. Let $\Q'=\C_M^4(\Q)$, with $M=(x_m,y_m)$. Via CAS, we obtain the following propositions.  

%By setting $\A(T^3(T_0))=\A(T_0)$

\begin{proposition}
Starting from the square $\Q$, $s=1$ occurs when $M=(x_m,y_m)$ satisfies:
\begin{align*}
& \;\;\;\; 15 x_m^8 -68 x_m^6 y_m^2 + 90 x_m^4 y_m^4 - 68 x_m^2 y_m^6 + 15 y_m^8 - 64 x_m^6  
 + 64 x_m^4 y_m^2+\\ 
 & 64 x_m^2 y_m^4 - 64 y_m^6 + 98 x_m^4 + 52 x_m^2 y_m^2 + 98 y_m^4 
 - 64 x_m^2 - 64 y_m^2 + 15=0
  %\label{eq:square}
\end{align*}
\end{proposition}

%Parallel to \cref{prop:alpha0-3}: 

\begin{proposition}\label{prop:alphazero}
Starting from the square $\Q$, $\alpha=0$ when $M=(x_m,y_m)$ satisfies:
\[x_my_m(x_m^2- y_m^2)=0\]
\end{proposition}

\section{Regular polygons, $n \geq 3$}
\label{sec:all-n}
In this section we assume $\P$ is a regular $n$-gon, $n{\geq}3$. Without loss of generality, let the centroid $O=(0,0)$ and the first vertex $P_1=(1,0)$. \cref{fig:regions3456} illustrates the partitioning of the plane into area-contracting and area-expanding regions by $n$ applications of the map, for $n=3,4,5,6$.

\begin{proposition}
As $M$ approaches a sideline of $\P$, $s$ approaches infinity.
\label{prop:inf}
\end{proposition}

The proof below was kindly contributed by a referee.

\begin{proof}
If $M$ is on a sideline of $\P$, say $P_i  P_{i+1}$, then $P_i'$ is at infinity (since $P_i'$ is the intersection of the bisectors $M P_i$ and $M P_{i+1}$), but the rest of the vertices $P_j'$ are finite. If a polygon $\P$ has a vertex at infinity, it will remain at infinity under the map $\C_M$.
 \end{proof}
\begin{figure}
    \centering
    \includegraphics[width=\textwidth]{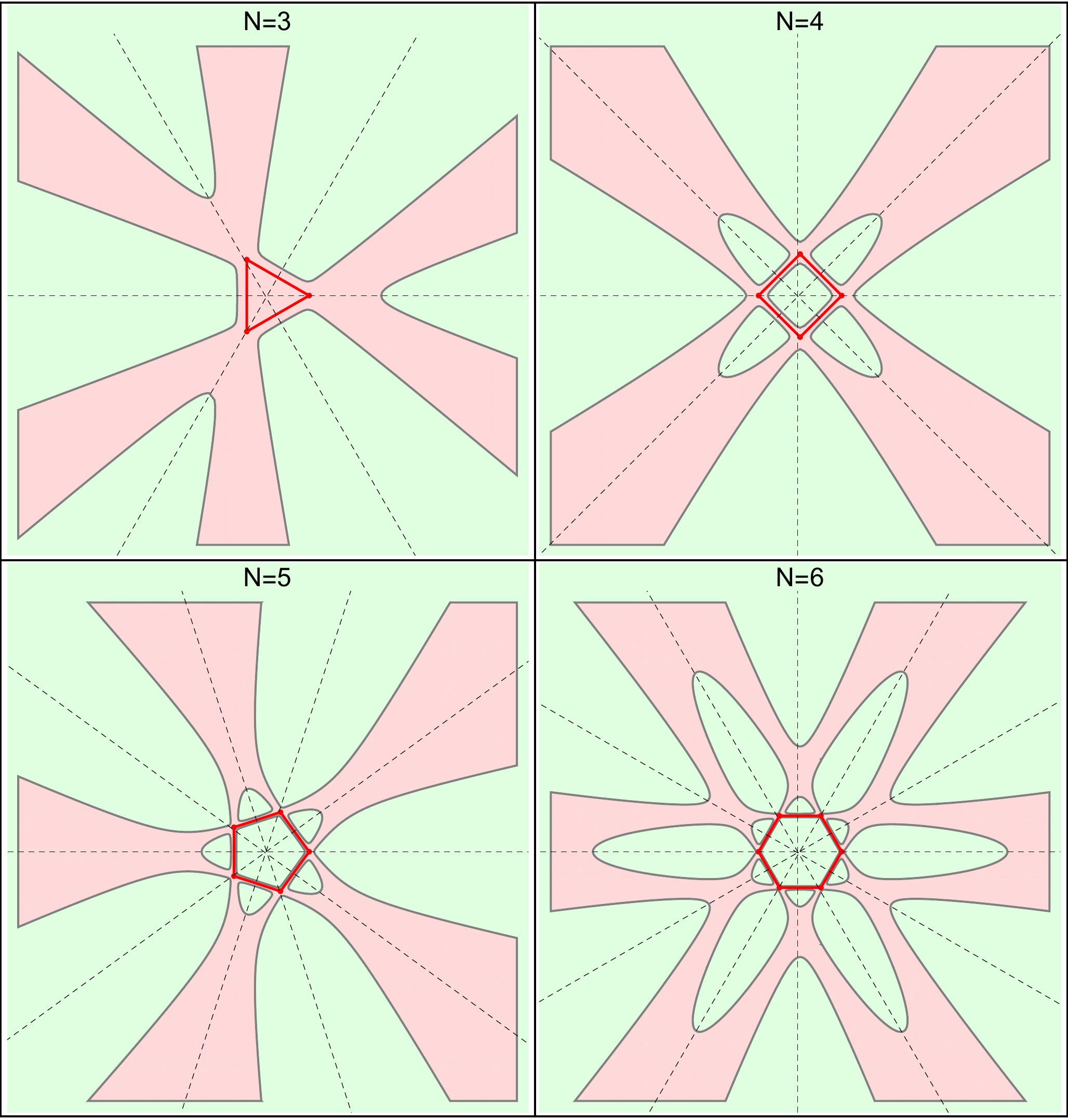}
    \caption{Area contraction (green) and expansion (red) regions for the circumcenter map applied to a regular triangle (top left), square (top right), pentagon (bottom left) and hexagon (bottom right). Notice that in all but in the $n=3$, an area contracting region exists interior to the original polygon. Also shown are the zero-rotation lines (dashed black).}
    \label{fig:regions3456}
\end{figure}

Let $\P'=\C_M^n(\P)$. Let $\alpha$ be the angle of rotation of the similarity that takes $\P$ to $\P'$.

\begin{conjecture}
The locus of $M$ such that $\alpha=0$ is the union of $n$ lines along directions $k \pi/n$, $k=0,\ldots,n-1$.
\end{conjecture}

To facilitate region counting, in \cref{fig:n3to11} the plane is compactified into a single hemisphere, via stereographic projection. \cref{tab:regions} shows the counts of area-contracting regions. This suggests:

\begin{table}
\begin{tabular}{|c|c|c|c|c|}
\hline
n & interior & non-compact & compact & total \\
\hline
3 & 0 & 2n & 0 & 2n \\
4 & 1 & n & n & 1+2n \\
5 & 1 & 2n & n & 1+3n \\
6 & 1 & n & 2n & 1+3n \\
7 & 1 & 2n & 2n & 1+4n \\
8 & 1 & n & 3n & 1+4n \\
9 & 1 & 2n & 3n & 1+5n \\
10 & 1 & n & 4n & 1+5n \\
11 & 1 & 2n & 4n & 1+6n\\
\hline
\end{tabular}
\caption{Region count according to \cref{fig:n3to11}}.
\label{tab:regions}
\end{table}

\begin{conjecture}
There is a single connected area-expanding region. Let $k$ denote the number of area-contracting connected regions. Then:
\[k=\begin{cases}
\;\mbox{odd}~n:\;\;r^*+n(n+1)/2\\
\mbox{even}~n:~1+n^2/2\\
\end{cases}
\]
where $r^*=0$ if $n=3$, and $1$ otherwise. 
\label{conj:k}
\end{conjecture}

Experiments suggest:

\begin{conjecture}
Given an $n$, the number of area-contracting regions for a simple $n$-gon $\Q$ (no self-intersections), is
maximal if $\Q$ is regular.
\end{conjecture}

%\begin{figure}
%    \centering
    %\includegraphics[width=\textwidth]{pics_130_n3456_3d_black.png}
    %\caption{Compactification of the area contraction (green) and expansion (red) regions via inverse stereographic projection. From left-to-right: $n=3,4,5,6$ (top shows the hemisphere with the origin at its center; bottom is the other side of the stereographic sphere, whose center or south pole is the line at infinity). Note that if $n$ is odd (resp. even), $4n$ (resp. $2n$) boundary lines are incident on the south pole.}
%    \label{fig:n3456-stereo}
%\end{figure}

\begin{figure}
    \centering
    \includegraphics[width=\textwidth]{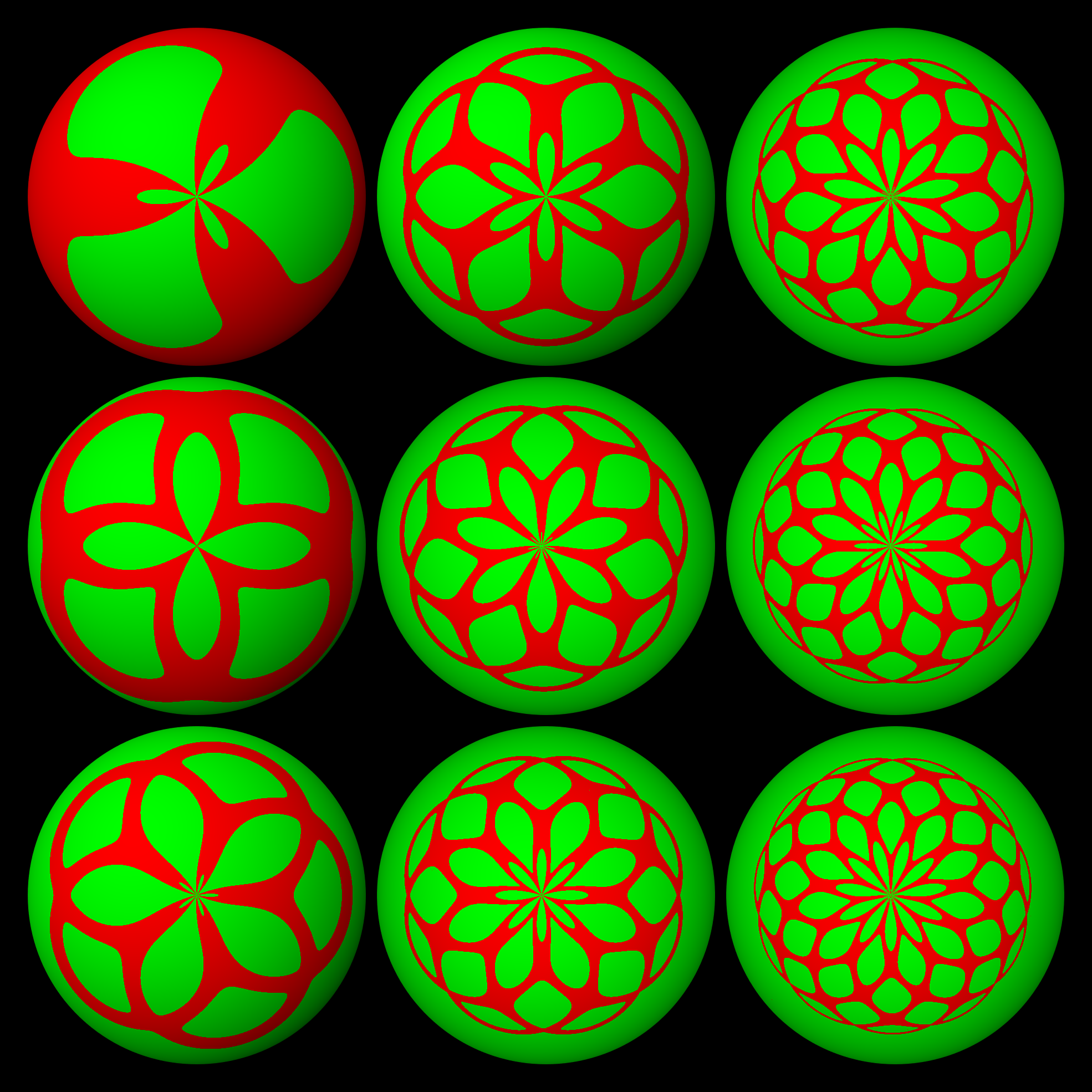}
     \caption{Area-expansion (red) and area-contraction (green) regions for regular $n$-gons, compactified (via stereographic projection) to a single hemisphere; the south pole (center) is ``infinity''. From top-to-bottom, left-to-right, $N=3,\ldots,11$. For interactive examples, see \href{https://bit.ly/3Pl2tmz}{\texttt{bit.ly/3Pl2tmz}} and \cite{mcdonald2022}.}
    \label{fig:n3to11}
\end{figure}

\section{conclusion}
\label{sec:conclusion}
A  video walk-through of our experiments appears in \cite{reznik2021-circum-youtube}. A gallery of interactive experiments appears in  \href{https://bit.ly/3Pl2tmz}{\texttt{bit.ly/3Pl2tmz}} with implementation details in \cite{mcdonald2022}.

See \cite{mcdonald2022} for a gallery of 
interactive experiments with the circumcenter map.

The circumcenter map can be generalized to the $X_k$-map, where $X_k$ is some triangle center (see \cite{kimberling1993_rocky}). For example, the $X_2$-map sends a polygon to one with vertices at the barycenters of $M P_i P_{i+1}$ of vertices of a given polygon. In such a case, an iteration produces a sequence of ever-shrinking polygons which converges to $M$. If the starting polygon is a triangle, a few notable cases include: (i) the $X_4$-map (orthocenter) is area preserving for all $M$, and the sequence of triangles tends to an infinite line; (ii) the $X_{16}$-map (second isodynamic point) induces regions of the plane such that 3 applications of the map are the identity (no rotation and no scaling), see \cite{reznik2021-wolfram-iso}. A question not addressed here, is whether a certain $X_k$-map is integrable in the sense of \cite{maxim2021-centroaffine,ovsienko2010-pentagram}.

\section*{Acknowledgements}
\noindent We would like to thank Sergei Tabachnikov and Richard Schwartz for discussions during the early experimental results, and Darij Grinberg for contributing a proof to \cref{cor:similar}. We thank Mark Helman for relating this phenomenon to the 1940 result by Stewart \cite{stewart1940}. We are indebted to Wolfram Communities for inviting us to post a short description of our experimental results, see \cite{reznik2021-wolfram}. Finally, we thank one of the referees for suggesting the inclusion of \cref{prop:alphazero} and 
for contributing the proof of \cref{prop:inf}.

\appendix

%\section{Proof to  invariant sum of cosines transformation}
%\input{130_proof_thm_scale}
%\label{app:Proof_Scale}

%\section{Proof to invariant product of cosines transformation}
%\input{140_proof_thm_scale2}
%\label{app:Proof_Scale_Product}
\section{Explicit Circumcenter Map}
\label{app:map-tri}
Let $\P$ be a generic $n$-gon with vertices  $(x_i,y_i)$,  $i=1,\ldots,n$, and $M=(x_m,y_m)$. The $M$-circumcenter map $\C_M(\P)$ yields a new polygon with vertices $(p_i,q_i)$ given by:
  {\small
 \begin{align*}
   p_{i}&=\frac{(  y_{i+1} -  y_m)   y_{i}^2 + \rho_i   y_{i} +   y_{i+1}^2  y_m + (  x_i^2 -  x_m^2 -  y_m^2)   y_{i+1} + (  x_{i+1}^2-  x_i^2 )  y_m}
   {2(x_m-    x_{i+1}  )   y_{i} + 2(    x_i -    x_m)   y_{i+1} + 2    (     x_{i+1}-x_i) y_m}\\
   q_{i}&=\frac{( x_m -   x_{i+1})   x_i^2 - \rho_i   x_i -   x_{i+1}^2  x_m + ( x_m^2 +  y_m^2 -   y_{i}^2 )   x_{i+1} +  x_m (  y_{i}^2 -   y_{i+1}^2)}{2(   y_{i+1} -    y_m)   x_i + 2(y_m-    y_{i}  )   x_{i+1} + 2   (y_i-  y_{i+1}  )x_m}
 \end{align*}
 }
  where $\rho_i=|M|^2-|B|^2=
 x_m^2+ y_m^2  -    x_{i+1}^2 - y_{i+1}^2  $.
 
Likewise, the inverse the $M$-circumcenter map $\C_M^{-1}(\P)$ yields a new polygon with vertices $(u_i,v_i)$ given by:
 
{\small
\begin{align*}
u_i&=\frac{r_i' x_m
+ 2(   y_{i+1} -    y_{i}) ( x_{i+1} -  x_{i}) y_m + 2(  y_{i+1} -  y_{i}) ( x_i  y_{i+1} -  x_{i+1}  y_{i})}{r_i}\\
v_i&= \frac{- r_i' y_m +2 (  y_{i+1} -  y_{i}) ( x_{i+1} -  x_{i}) x_m 
- 2 ( x_{i+1} -  x_{i}) ( x_i  y_{i+1} -  x_{i+1}  y_{i})}{r_i}
\end{align*}
}
where $r_i=( x_{i+1} -  x_i)^2 + ( y_{i+1} -  y_{i})^2$, and $r_i'=(x_{i+1} - x_i)^2 -  (y_{i+1} - y_i)^2$.

%\section{The circumcenter map explicitly}
%\input{220_app_circumcentermap}

%\section{Table of Symbols}
%\input{120_app_symbols}
%\label{app:symbols}

\bibliographystyle{hacm}
\raggedright
\setlength{\baselineskip}{13pt}
\bibliography{999_refs,999_refs_rgk}

%\section{Outline}
%\input{900_outline}

\end{document}